\documentclass[11pt]{amsart}

\usepackage[T1]{fontenc}
\usepackage{lmodern}
\usepackage{microtype}
\usepackage[margin=1in]{geometry}

\usepackage{amsmath,amssymb,mathtools}
\newcommand{\ostar}{\circledast}

\usepackage{graphicx}
\usepackage{xcolor}
\usepackage{tikz}
\usepackage{subcaption}
\usepackage{float}
\usepackage{flafter}
\usepackage{array,tabularx}

\usepackage[hidelinks]{hyperref}
\usepackage[nameinlink,capitalise,noabbrev]{cleveref}
\numberwithin{equation}{section}
\usetikzlibrary{calc,positioning,fit,shapes.geometric,backgrounds}

\theoremstyle{plain}
\newtheorem{theorem}{Theorem}[section]
\newtheorem{lemma}[theorem]{Lemma}
\newtheorem{proposition}[theorem]{Proposition}

\newtheorem{conjecture}[theorem]{Conjecture}

\theoremstyle{definition}
\newtheorem{definition}[theorem]{Definition}

\theoremstyle{remark}

\definecolor{AuditNavy}{HTML}{17324D}
\definecolor{AuditBlue}{HTML}{2A6690}
\definecolor{AuditPaleBlue}{HTML}{E8EEF3}
\definecolor{AuditPaleGold}{HTML}{A2C96D}
\definecolor{AuditGold}{HTML}{B8892B}
\definecolor{AuditRed}{HTML}{9B2C2C}
\definecolor{AuditPaleRed}{HTML}{A4B7B1}
\definecolor{AuditBrown}{HTML}{8B5B2B}
\definecolor{AuditGray}{HTML}{64748B}

\tikzset{
  ordinary vertex/.style={
    circle,draw=AuditNavy,fill=AuditPaleBlue,
    minimum size=3.6pt,inner sep=0pt,line width=.45pt
  },
  attachment vertex/.style={
    circle,double,draw=AuditBrown,fill=AuditPaleGold,
    minimum size=5.6pt,inner sep=0pt,line width=.55pt,
    double distance=.55pt
  },
  center vertex/.style={
    circle,double,draw=AuditBlue,fill=white,
    minimum size=5.8pt,inner sep=0pt,line width=.55pt,
    double distance=.55pt
  },
  root vertex/.style={
    circle,double,draw=AuditRed,fill=AuditPaleRed,
    minimum size=6.4pt,inner sep=0pt,line width=.65pt,
    double distance=.65pt
  },
  child port/.style={
    circle,double,draw=AuditBlue,fill=AuditPaleBlue,
    minimum size=6.4pt,inner sep=0pt,line width=.65pt,
    double distance=.65pt
  },
  internal edge/.style={draw=AuditGray,line width=.48pt},
  cycle edge/.style={draw=AuditNavy,line width=.60pt},
  bridge edge/.style={draw=AuditBrown,line width=1.05pt},
  root cycle edge/.style={draw=AuditBlue,line width=1.45pt},
}

\tikzset{
  pics/gadgetB/.style={
    code={
      \coordinate (-root) at (0:1.05);
      \draw[cycle edge]
        (0:1.05)--(-40:1.05)--(-80:1.05)--(-120:1.05)--
        (-160:1.05)--(-200:1.05)--(-240:1.05)--(-280:1.05)--
        (-320:1.05)--cycle;
      \draw[internal edge] (-40:1.05)--(-200:1.05);
      \draw[internal edge] (-80:1.05)--(-280:1.05);
      \draw[internal edge] (-120:1.05)--(-240:1.05);
      \draw[internal edge] (-160:1.05)--(-320:1.05);
      \foreach \ang in {-40,-80,-120,-160,-200,-240,-280,-320}{
        \node[ordinary vertex] at (\ang:1.05) {};
      }
      \node[attachment vertex] at (0:1.05) {};
    }
  },
}

\tikzset{
  pics/blockBJoin/.style={
    code={
      \pic (Bone) at (-2.70,0) {gadgetB};
      \pic[rotate=180,transform shape] (Btwo) at (2.70,0) {gadgetB};
      \coordinate (-attachmentone) at (Bone-root);
      \coordinate (-attachmenttwo) at (Btwo-root);
      \coordinate (-root) at (0,0);
      \draw[bridge edge] (-attachmentone)--(-root);
      \draw[bridge edge] (-root)--(-attachmenttwo);
      \node[root vertex] at (-root) {};
    }
  },
}

\tikzset{
  pics/blockBltimes/.style={
    code={
      \pic[scale=.55,transform shape] (Bcore) at (-1.25,.35) {gadgetB};
      \coordinate (-attachment) at (Bcore-root);
      \coordinate (-center) at (.15,.35);
      \coordinate (-leaf) at (1.15,.35);
      \coordinate (-root) at (.15,-.85);
      \draw[bridge edge] (-attachment)--(-center)--(-leaf);
      \draw[bridge edge] (-center)--(-root);
      \node[center vertex] at (-center) {};
      \node[attachment vertex] at (-leaf) {};
      \node[root vertex] at (-root) {};
    }
  },
}

\title[Domination--packing ratios in connected subcubic graphs]
{New lower bounds on domination--packing ratios\\ in connected subcubic and cubic graphs}

\author{JiSun Huh}
\address{Department of Mathematics, University of Seoul, Seoul 02504, South Korea}
\email{hyunyjia@yonsei.ac.kr}

\author{Juho Kim}
\address{Department of Economics, University of Seoul, Seoul 02504, South Korea}
\email{kjuho0403@uos.ac.kr}
\date{}


\begin{document}

\begin{abstract}
For a graph \(G\), let \(\gamma(G)\) and \(\rho(G)\) denote its
domination number and packing number, respectively. Let
\(c_{\mathrm{sub}}\) and \(c_{\mathrm{cub}}\) denote the respective
limsups of \(\gamma(G)/\rho(G)\) over connected subcubic and connected
cubic graphs as \(\rho(G)\to\infty\). We prove
\[
  c_{\mathrm{sub}}\geq\frac{13}{6},
  \qquad
  c_{\mathrm{cub}}\geq\frac{17}{8},
\]
by constructing two explicit binary branching families. The connected
noncubic subcubic graphs \(\widehat B_t^\star\) satisfy
\[
 |V(\widehat B_t^\star)|=76\cdot2^t-12,\qquad
 \gamma(\widehat B_t^\star)=26\cdot2^t-4,\qquad
 \rho(\widehat B_t^\star)=12\cdot2^t-2,
\]
whereas the connected cubic graphs \(\widehat B_t^\bullet\) satisfy
\[
 |V(\widehat B_t^\bullet)|=108\cdot2^t-14,\qquad
 \gamma(\widehat B_t^\bullet)=34\cdot2^t-4,\qquad
 \rho(\widehat B_t^\bullet)=16\cdot2^t-2.
\]
The constructions use the same binary connector composition and
closing lemma, with different connectors and initial assemblies. As a consequence,
both families give unbounded additive violations of
\(\gamma(G)\leq2\rho(G)+1\), disproving the proposed inequality even
for connected cubic graphs. 
\end{abstract}

\maketitle

\section{Introduction}

Throughout, all graphs are finite, simple, and undirected.
For a graph \(G\), write \(\Delta(G)\) for its maximum degree, and
write \(d_G(u,v)\) for the distance between vertices \(u\) and \(v\).
A graph is \emph{subcubic} if \(\Delta(G)\leq3\), and \emph{cubic}
if every vertex has degree \(3\).
For a vertex \(v\in V(G)\), let \(N_G(v)\) and
\(N_G[v]=N_G(v)\cup\{v\}\) denote its open and closed neighborhoods,
respectively.  A set \(D\subseteq V(G)\) is \emph{dominating} if
\(D\cap N_G[v]\neq\varnothing\) for every \(v\in V(G)\); the minimum
cardinality of such a set is the \emph{domination number} \(\gamma(G)\).
A set \(P\subseteq V(G)\) is a \emph{packing} if
\(d_G(u,v)\geq3\) for all distinct \(u,v\in P\), equivalently if the
closed neighborhoods of its vertices are pairwise disjoint; the
maximum cardinality of such a set is the \emph{packing number}
\(\rho(G)\).

The elementary inequality
\[
  \rho(G)\leq\gamma(G)
\]
follows because every dominating set must meet each of the pairwise
disjoint closed neighborhoods centered at the vertices of a packing.
In the opposite direction, no bound on \(\gamma(G)\) in terms of
\(\rho(G)\) alone is possible for arbitrary graphs.  For the Cartesian
product \(K_n\square K_n\) of two complete graphs, any two vertices are at distance at most
\(2\), while fewer than \(n\) vertices miss a row and a column and a full
row dominates.  Hence
\[
  \rho(K_n\square K_n)=1
  \qquad\text{and}\qquad
  \gamma(K_n\square K_n)=n.
\]

A central line of research is therefore to identify graph classes for
which the ratio \(\gamma(G)/\rho(G)\) is bounded and to determine the
best class-dependent constant \cite{BurgerHenningVanVuuren2009};
results are known for trees, chordal
subclasses, and various sparse, bounded-width, geometric, and
low-degree classes; see
\cite{BonamyCsikosGujgiczerYuditsky2025,GomezGutierrez2025}
for recent discussions and further references.
Recent work improved the bound for planar graphs and established
bounded ratios for chordal-bipartite and homogeneously orderable
graphs \cite{DuczGujgiczer2026}.  This was followed by further
improvements for planar and unit-disk graphs
\cite{CamesVanBatenburg2026}.

For graphs with no isolated vertices, Henning, L\"owenstein, and
Rautenbach proved
\[
  \gamma(G)\leq\Delta(G)\rho(G),
\]
and they established the sharper inequality
\(\gamma(G)\leq2\rho(G)\) for claw-free subcubic graphs
\cite{HenningLowensteinRautenbach2011}.
They also proposed the following conjecture.

\begin{conjecture}[Henning--L\"owenstein--Rautenbach
  {\cite[Conjecture~3]{HenningLowensteinRautenbach2011}}]
\label{conj:HLR}
Let \(G\) be a connected subcubic graph.
\begin{enumerate}
\item The inequality
\[
  \gamma(G)\leq 2\rho(G)+1
\]
holds.
\item Equality holds if and only if
\[
  G\in\{H_1,H_2,H_3\},
\]
where \(H_1,H_2,H_3\) are the three graphs specified by Henning,
L\"owenstein, and Rautenbach.
\end{enumerate}
\end{conjecture}

The graphs \(H_1\) and \(H_2\) are the two exceptional cubic graphs
of order \(8\), and \(H_3\) is the Petersen graph.
Each of them satisfies
\[
  \gamma(H_i)=3,\qquad \rho(H_i)=1,
\]
and hence \(\gamma(H_i)=2\rho(H_i)+1\).  They are shown in
Figure~\ref{fig:conjecture-equality-graphs}.  Henning, L\"owenstein, and
Rautenbach also proved that, for a connected subcubic graph \(G\),
\[
  \gamma(G)=3\rho(G)
  \quad\Longleftrightarrow\quad
  G\in\{H_1,H_2,H_3\}
\]
\cite[Theorem~4]{HenningLowensteinRautenbach2011}.
The conjectured bound is also known for several restricted classes
of subcubic graphs. Henning, Maniya, and Pradhan proved that every
connected subcubic graph with no induced \(P_5\) satisfies
\(\gamma(G)\leq2\rho(G)\)
\cite[Theorem~3]{HenningManiyaPradhan2026}.
Guti\'errez and Paul proved
that every bridgeless claw-free cubic graph \(G\) satisfies
\[
  \gamma(G)
  \leq
  \frac{7}{4}\rho(G)+\frac{5}{6}
\]
\cite[Theorem~5]{GutierrezPaul2026}.

Conjecture~\ref{conj:HLR} contains two logically distinct assertions:
the inequality in part~(1) and the equality classification in part~(2).
Bonamy, Csik\'os, Gujgiczer, and Yuditsky (BCGY) constructed, for every
integer \(i\geq1\), a graph \(G_i\) of order \(6i+2\) such that
\[
  \rho(G_i)=i,
  \qquad
  \gamma(G_i)=2i+1.
\]
\cite[Theorem~1.12]
  {BonamyCsikosGujgiczerYuditsky2025}
Thus this construction attains equality in part~(1) for every \(i\),
while disproving the equality classification in part~(2).
Figure~\ref{fig:conjecture-equality-graphs} illustrates \(G_3\).

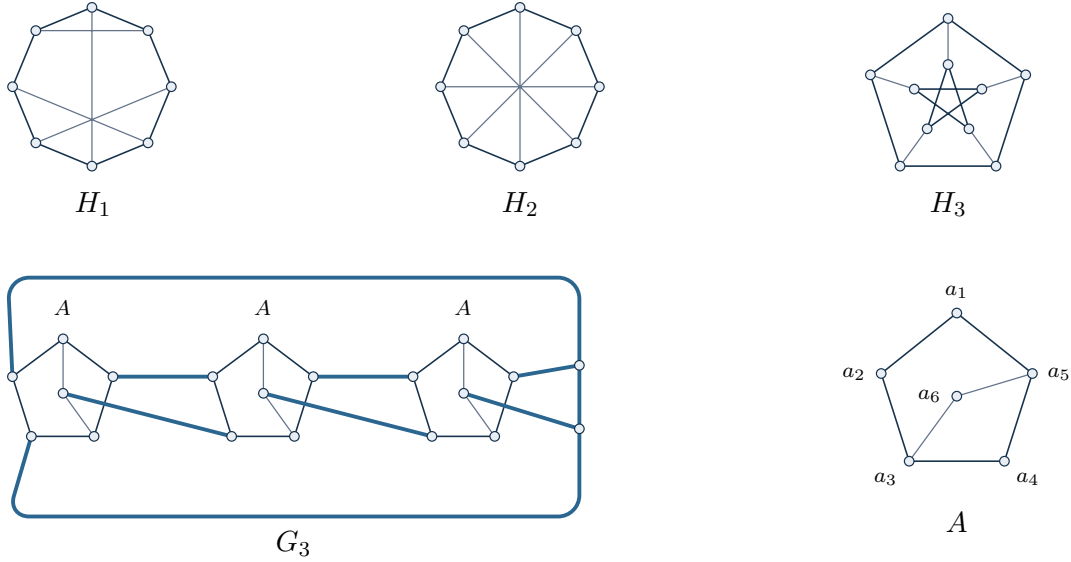
\begin{figure}
  \centering
  \begin{minipage}[t]{.98\textwidth}
    \centering
  \begin{minipage}[b]{.30\textwidth}
    \centering
    \begin{tikzpicture}[scale=1]
      \foreach \v/\ang in
        {0/45,1/0,4/-45,6/-90,5/-135,7/180,3/135,2/90}{
        \coordinate (h1-\v) at (\ang:1.05);
      }
      \draw[cycle edge]
        (h1-0)--(h1-1)--(h1-4)--(h1-6)--(h1-5)--(h1-7)--
        (h1-3)--(h1-2)--cycle;
      \draw[internal edge]
        (h1-0)--(h1-3) (h1-1)--(h1-5)
        (h1-2)--(h1-6) (h1-4)--(h1-7);
      \foreach \v in {0,...,7}{
        \node[ordinary vertex] at (h1-\v) {};
      }
    \end{tikzpicture}
    \par\smallskip\centering\(H_1\)
  \end{minipage}\hfill
  \begin{minipage}[b]{.30\textwidth}
    \centering
    \begin{tikzpicture}[scale=1]
      \foreach \v/\ang in
        {0/90,1/45,4/0,7/-45,3/-90,6/-135,5/180,2/135}{
        \coordinate (h2-\v) at (\ang:1.05);
      }
      \draw[cycle edge]
        (h2-0)--(h2-1)--(h2-4)--(h2-7)--(h2-3)--(h2-6)--
        (h2-5)--(h2-2)--cycle;
      \draw[internal edge]
        (h2-0)--(h2-3) (h2-1)--(h2-6)
        (h2-2)--(h2-7) (h2-4)--(h2-5);
      \foreach \v in {0,...,7}{
        \node[ordinary vertex] at (h2-\v) {};
      }
    \end{tikzpicture}
    \par\smallskip\centering\(H_2\)
  \end{minipage}\hfill
  \begin{minipage}[b]{.30\textwidth}
    \centering
    \begin{tikzpicture}[scale=1]
      \foreach \v/\ang in {0/90,1/18,4/-54,6/-126,2/162}{
        \coordinate (h3-\v) at (\ang:1.08);
      }
      \foreach \v/\ang in {3/90,5/18,8/-54,7/-126,9/162}{
        \coordinate (h3-\v) at (\ang:.47);
      }
      \draw[cycle edge]
        (h3-0)--(h3-1)--(h3-4)--(h3-6)--(h3-2)--cycle;
      \draw[cycle edge]
        (h3-3)--(h3-7)--(h3-5)--(h3-9)--(h3-8)--cycle;
      \draw[internal edge]
        (h3-0)--(h3-3) (h3-1)--(h3-5) (h3-4)--(h3-8)
        (h3-6)--(h3-7) (h3-2)--(h3-9);
      \foreach \v in {0,...,9}{
        \node[ordinary vertex] at (h3-\v) {};
      }
    \end{tikzpicture}
    \par\smallskip\centering\(H_3\)
  \end{minipage}

  \end{minipage}

  \par\vspace{1em}
  \begin{minipage}[t]{.64\textwidth}
    \centering
    \vspace{0pt}
  \begin{tikzpicture}[scale=1,transform shape]
    \foreach \j/\xx in {1/0,2/2.65,3/5.30}{
      \coordinate (b\j-3) at (\xx+.06,.27);
      \coordinate (b\j-4) at (\xx+.31,-.52);
      \coordinate (b\j-5) at (\xx+1.14,-.52);
      \coordinate (b\j-1) at (\xx+1.39,.27);
      \coordinate (b\j-2) at (\xx+.73,.77);
      \coordinate (b\j-6) at (\xx+.73,.05);
      \draw[cycle edge]
        (b\j-3)--(b\j-4)--(b\j-5)--(b\j-1)--
        (b\j-2)--cycle;
      \draw[internal edge]
        (b\j-5)--(b\j-6)--(b\j-2);
      \node[font=\scriptsize] at (\xx+.72,1.18) {\(A\)};
    }
    \draw[root cycle edge]
      (b1-1)--(b2-3) (b1-6)--(b2-4)
      (b2-1)--(b3-3) (b2-6)--(b3-4);
    \coordinate (u1) at (7.56,.42);
    \coordinate (u2) at (7.56,-.42);
    \draw[root cycle edge] (u1)--(u2);
    \draw[root cycle edge] (u1)--(b3-1) (u2)--(b3-6);
    \draw[root cycle edge,rounded corners=8pt]
      (b1-3)--(0,1.58)--(7.56,1.58)--(u1);
    \draw[root cycle edge,rounded corners=8pt]
      (b1-4)--(0,-1.58)--(7.56,-1.58)--(u2);
    \foreach \j in {1,2,3}{
      \foreach \h in {1,...,6}{
        \node[ordinary vertex] at (b\j-\h) {};
      }
    }
    \node[ordinary vertex] at (u1) {};
    \node[ordinary vertex] at (u2) {};
    \node at (3.78,-1.94) {\(G_3\)};
  \end{tikzpicture}
  \end{minipage}\hfill
  \begin{minipage}[t]{.30\textwidth}
    \centering
    \vspace{0pt}
    \begin{tikzpicture}[scale=1,transform shape]
      \coordinate (fc1) at (0,1.15);
      \coordinate (fc2) at (-1.00,.35);
      \coordinate (fc3) at (-.63,-.81);
      \coordinate (fc4) at ( .63,-.81);
      \coordinate (fc5) at (1.00,.35);
      \coordinate (fc6) at (0,.05);
      \draw[cycle edge] (fc1)--(fc2)--(fc3)--(fc4)--(fc5)--cycle;
      \draw[internal edge] (fc3)--(fc6)--(fc5);
      \foreach \i in {1,...,6}
        \node[ordinary vertex] at (fc\i) {};
      \node[font=\scriptsize,above=2pt of fc1] {\(a_1\)};
      \node[font=\scriptsize,left=2pt of fc2] {\(a_2\)};
      \node[font=\scriptsize,below left=1pt of fc3] {\(a_3\)};
      \node[font=\scriptsize,below right=1pt of fc4] {\(a_4\)};
      \node[font=\scriptsize,right=2pt of fc5] {\(a_5\)};
      \node[font=\scriptsize,left=2pt of fc6] {\(a_6\)};
      \node at (0,-1.62) {\(A\)};
    \end{tikzpicture}
  \end{minipage}
  \caption{The three proposed equality graphs \(H_1,H_2,H_3\), the
    BCGY equality example \(G_3\), and its six-vertex unit \(A\).}
  \label{fig:conjecture-equality-graphs}
  \label{fig:graph-A}
\end{figure}

To formulate our results asymptotically, let \(c_{\mathrm{sub}}\) and
\(c_{\mathrm{cub}}\) denote the respective limsups of
\(\gamma(G)/\rho(G)\) over connected subcubic and connected cubic
graphs as \(\rho(G)\to\infty\). Our bounds are
\begin{equation}
\label{eq:main-asymptotic-bounds}
 c_{\mathrm{sub}}\geq\frac{13}{6},
 \qquad
 c_{\mathrm{cub}}\geq\frac{17}{8}.
\end{equation}
The lower bounds are witnessed by two explicit binary branching
families, indexed by \(t\geq0\), whose packing numbers tend to infinity.
The connected noncubic subcubic family \(\widehat B_t^\star\) has
domination--packing ratio tending to \(13/6\), while the connected
cubic family \(\widehat B_t^\bullet\) has ratio tending to \(17/8\).
Moreover, their gaps above \(2\rho\) are explicit:
\begin{equation}
\label{eq:intro-additive-gaps}
 \gamma(\widehat B_t^\star)
 -2\rho(\widehat B_t^\star)
 =
 \gamma(\widehat B_t^\bullet)
 -2\rho(\widehat B_t^\bullet)=2^{t+1}.
\end{equation}
Thus, unlike the BCGY family, both families strictly violate the
inequality in part~(1), with an unbounded additive violation. In
particular, the family \(\widehat B_t^\bullet\) disproves part~(1)
even within connected cubic graphs.

The two families follow a common rooted-profile scheme. Starting from
the seed \(B\), a base transfer reaches a balanced profile class;
binary recursion preserves that class under the same scalar
recurrence, and a final closure removes the boundary. Thus the exact
formulas follow from the profile scheme rather than from separate
calculations.

Section~\ref{sec:boundary-states} develops the boundary states and
rooted profiles, treats \(B\), the rooted operations, and the two
derived rooted graphs \(B^{\Join},B^{\ltimes}\), and records the common
connector and closing lemmas.
Section~\ref{sec:star-family} constructs \(\widehat B_t^\star\), and
Section~\ref{sec:bullet-family} constructs \(\widehat B_t^\bullet\).
Section~\ref{sec:consequences} gives the concluding remarks and
formulates three revised versions of Conjecture~\ref{conj:HLR}.

\section{Boundary states, rooted profiles, and basic operations}
\label{sec:boundary-states}

This section develops the common local language for the two
constructions.  We first define the boundary states and rooted
profiles, then compute the seed profile of \(B\) and derive the rooted
graphs \(B^{\Join}\) and \(B^{\ltimes}\) through rooted operations, and
finally record the connector composition and closing lemma used by
both families.

\subsection{Boundary states and rooted profiles}
A \emph{rooted graph} \((G,x)\) consists of a graph \(G\) and a
distinguished vertex \(x\), which we regard as a one-vertex boundary.
The boundary decoration is auxiliary and is forgotten after the
relevant copies are joined.

For \(X\subseteq V(G)\), put
\[
  d_G(x,X)\coloneqq\min_{v\in X}d_G(x,v),
  \qquad
  N_G[X]\coloneqq\bigcup_{v\in X}N_G[v],
\]
where \(d_G(x,\varnothing)=\infty\).
For such \(X\), we also define the boundary-state operators
\[
  \sigma_x^{\mathrm D}(X)\coloneqq\min\{2,d_G(x,X)\},
  \qquad
  \sigma_x^{\mathrm P}(X)\coloneqq\min\{3,d_G(x,X)\}.
\]
For domination, an \emph{\(x\)-boundary domination candidate} is a
set \(D\subseteq V(G)\) satisfying
\[
  V(G)\setminus\{x\}\subseteq N_G[D].
\]
Its boundary state is \(\sigma_x^{\mathrm D}(D)\): state \(0\) means
\(x\in D\), state \(1\) means \(x\notin D\) but
\(N_G(x)\cap D\ne\varnothing\), and state \(2\) means
\(N_G[x]\cap D=\varnothing\). Thus \(D\) dominates all of \(G\) in
states \(0\) and \(1\), but leaves \(x\) undominated in state \(2\).

For a packing \(P\subseteq V(G)\), the state
\(\sigma_x^{\mathrm P}(P)\) records
its truncated distance from \(x\). The states \(0,1,2,3\) correspond,
respectively, to distance \(0,1,2\), and at least \(3\), with the
empty packing included in state \(3\).

For \(i\in\{0,1,2\}\), let \(\gamma_i(G,x)\) be the minimum
cardinality of an \(x\)-boundary domination candidate \(D\) with
\(\sigma_x^{\mathrm D}(D)=i\), using the value \(+\infty\) when no such candidate
exists. For \(i\in\{0,1,2,3\}\), let \(\rho_i(G,x)\) be the maximum
cardinality of a packing \(P\) with \(\sigma_x^{\mathrm P}(P)=i\), using the value
\(-\infty\) when no such packing exists. The \emph{rooted domination
profile} and the \emph{rooted packing profile} are
\[
  \gamma_\bullet(G,x)
  \coloneqq
  \bigl(
    \gamma_0(G,x),
    \gamma_1(G,x),
    \gamma_2(G,x)
  \bigr),
  \qquad
  \rho_\bullet(G,x)
  \coloneqq
  \bigl(
    \rho_0(G,x),
    \rho_1(G,x),
    \rho_2(G,x),
    \rho_3(G,x)
  \bigr).
\]

Since a boundary domination candidate dominates all of \(G\) exactly
in states \(0\) and \(1\), while every packing lies in one of the four
packing states, the ordinary parameters are recovered by
\[
  \gamma(G)=\min\{\gamma_0(G,x),\gamma_1(G,x)\},
  \qquad
  \rho(G)=\max_{0\leq i\leq3}\rho_i(G,x).
\]

For example, root the edge \(K_2\) at one endpoint \(x\). For
domination, selecting \(x\) gives state \(0\), and selecting the other
endpoint gives state \(1\); state \(2\) is infeasible. For packing,
the same two singletons give states \(0\) and \(1\). No packing has
distance exactly \(2\) from \(x\), whereas the empty packing has
distance \(\infty\) and hence state \(3\). Therefore
\begin{equation}\label{eq:rooted-edge-profiles}
  \gamma_\bullet(K_2,x)=(1,1,+\infty),
  \qquad
  \rho_\bullet(K_2,x)=(1,1,-\infty,0).
\end{equation}
Here \(+\infty\) and \(-\infty\) record infeasibility in the
minimization and maximization profiles, respectively; the state-\(3\)
packing value \(0\), by contrast, is attained by the empty packing.

When the distinguished root is fixed by the construction or is clear
from the context, we suppress it from the notation. Thus, for a rooted
graph \((H,x)\) whose root is understood,
\[
  \gamma_i(H)\coloneqq\gamma_i(H,x),\qquad
  \rho_i(H)\coloneqq\rho_i(H,x),
\]
and
\[
  \gamma_\bullet(H)\coloneqq\gamma_\bullet(H,x),\qquad
  \rho_\bullet(H)\coloneqq\rho_\bullet(H,x).
\]

The rooted profiles above are one-vertex instances of the standard
practice of encoding partial solutions by finitely many boundary states
in dynamic programming over graph decompositions
\cite{GarneroPaulSauThilikos2015,vanRooijBodlaenderEtAl2018}.
In our constructions, we combine these profiles using the fixed
connector graphs introduced in
Sections~\ref{sec:star-family} and \ref{sec:bullet-family}.

\subsection{\texorpdfstring{The seed graph \(B\)}{The seed graph B}}

Let \(B\) be the subcubic graph obtained from the cycle
\(C_9=b_1b_2\cdots b_9b_1\) by adding the four edges
\[
  b_2b_6,\qquad b_3b_8,\qquad b_4b_7,\qquad b_5b_9.
\]
Set \(b=b_1\) and regard \((B,b)\) as rooted at \(b\).
The root \(b\) has degree \(2\), while every other vertex has degree
\(3\). Since the defining \(9\)-cycle is spanning, \(B\) is connected
and bridgeless; see
Figure~\ref{fig:B-shapes}\subref{fig:block-B}.

\begin{lemma}\label{lem:profile-B}
The rooted domination and packing profiles of
\((B,b)\) are
\[
  \gamma_\bullet(B)=(3,3,3)
  \qquad\text{and}\qquad
  \rho_\bullet(B)=(2,1,1,1).
\]
\end{lemma}

\begin{table}
\centering
\caption{Exact boundary profiles and witnesses for the rooted seed \((B,b)\).}
\label{tab:B-profile-witnesses}
{\renewcommand{\arraystretch}{1.3}
\setlength{\tabcolsep}{4pt}
\begin{tabularx}{\linewidth}{
  |>{\centering\arraybackslash}p{1.2cm}||
   >{\centering\arraybackslash}p{1.6cm}|
   >{\centering\arraybackslash}X||
   >{\centering\arraybackslash}p{1.9cm}|
   >{\centering\arraybackslash}X|
}
\hline
\textbf{State}&
\(\boldsymbol{\gamma_i(B)}\)&
\textbf{Domination witness}&
\(\boldsymbol{\rho_i(B)}\)&
\textbf{Packing witness}
\\
\hline
\hline
\(0\)&
\(3\)&
\(\{b,b_5,b_8\}\)&
\(2\)&
\(\{b,b_7\}\)
\\
\hline
\(1\)&
\(3\)&
\(\{b_2,b_5,b_7\}\)&
\(1\)&
\(\{b_2\}\)
\\
\hline
\(2\)&
\(3\)&
\(\{b_3,b_4,b_5\}\)&
\(1\)&
\(\{b_3\}\)
\\
\hline
\(3\)&
{}&
{}&
\(1\)&
\(\{b_4\}\)
\\
\hline
\end{tabularx}}
\end{table}

\begin{proof}
Table~\ref{tab:B-profile-witnesses} gives the required domination upper bounds and
packing lower bounds. It remains to prove the matching reverse
inequalities.

For domination, states \(0\) and \(1\) follow from closed-neighborhood
counting. In state \(0\), the selected root covers only two nonroot
vertices, and one additional closed neighborhood covers at most four
nonroot vertices; thus two selected vertices cannot dominate all eight
required nonroot vertices. In state \(1\), two selected nonroot vertices
have closed neighborhoods whose union has size at most \(8<9\). Hence
\(\gamma_0(B),\gamma_1(B)\geq 3\).
For state \(2\), let
\[
  U=V(B)\setminus N_B[b]
   =\{b_3,b_4,b_5,b_6,b_7,b_8\}.
\]
Any two distinct vertices of \(U\) are at distance at most \(2\), so
their closed neighborhoods intersect. Since each has size \(4\), two
vertices of \(U\) dominate at most \(7\) vertices and hence cannot
dominate all eight nonroot vertices. Thus \(\gamma_2(B)\geq3\), and
therefore
\(\gamma_\bullet(B)=(3,3,3)\).

For packing state \(0\), the only vertices at distance at least \(3\)
from \(b\) are \(b_4\) and \(b_7\), which are adjacent. Thus at most
one vertex can be added to the selected root. In state \(1\), a selected
neighbor of the root is either \(b_2\) or \(b_9\), and every vertex of
\(B\) is within distance \(2\) of that neighbor. In each of states
\(2\) and \(3\), all
selected vertices lie in \(U\), and any two vertices of \(U\) are at
distance at most \(2\). Together with the witnesses in
Table~\ref{tab:B-profile-witnesses}, this
gives \( \rho_\bullet(B)=(2,1,1,1)\).
\end{proof}

\subsection{\texorpdfstring{Root extension and rooted join; the derived graphs
\(B^{\Join}\) and \(B^{\ltimes}\)}{Root extension and rooted join; the derived graphs B-Join and B-left-times}}

Having determined the profile of the seed, we now introduce the two
basic rooted operations that produce the derived rooted graphs used in the
branching constructions.

\begin{definition}
Let \((H,x)\) and \((K,y)\) be vertex-disjoint rooted graphs.  We also
regard \(K_1\) as rooted at its unique vertex.
\begin{enumerate}
\renewcommand{\labelenumi}{(\alph{enumi})}
\item The \emph{pendant-root extension} \((H^\uparrow,p)\) is obtained
by adding a new vertex \(p\) and the edge \(px\).

\item The \emph{rooted join} \((H\odot K,r)\) is obtained by identifying
\(x\) and \(y\) to a common root \(r\).

\item The \emph{symmetric rooted join} associated with \(H\) is
\[
  H^{\Join}\coloneqq
  H^\uparrow\odot H^\uparrow.
\]

\item The \emph{asymmetric rooted shape} associated with \(H\) is
\[
  H^{\ltimes}\coloneqq
  \bigl(H^\uparrow\odot K_1^\uparrow\bigr)^\uparrow.
\]
Here \(K_1^\uparrow\) is the rooted edge \(K_2\).
The root added by the final extension is denoted by \(u\).
\end{enumerate}

The symbols \(\Join\) and \(\ltimes\) abbreviate the displayed
compositions of \(\uparrow\) and \(\odot\), while indicating the shapes
of the resulting rooted graphs.
\end{definition}

For the seed \((B,b)\), these constructions give the rooted graphs
\(B^{\Join}\) and \(B^{\ltimes}\) shown in
Figure~\ref{fig:B-shapes}. In \(B^{\ltimes}\), denote by \(w\) the
nonroot leaf on the \(K_1^\uparrow\) side.

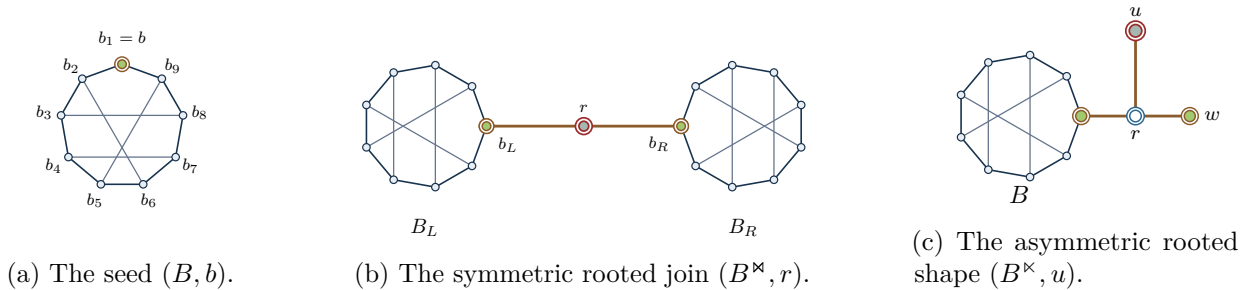
\begin{figure}
  \centering
  \begin{subfigure}[b]{.21\textwidth}
    \centering
    \begin{tikzpicture}[scale=.78,transform shape]
      \pic[rotate=90,transform shape] (B) {gadgetB};
      \node[font=\scriptsize,anchor=south] at (90:1.24) {\(b_1=b\)};
      \foreach \i/\ang in {2/130,3/170,4/-150,5/-110,
                             6/-70,7/-30,8/10,9/50}{
        \node[font=\scriptsize] at (\ang:1.34) {\(b_{\i}\)};
      }
      \path (0,-1.994)--(0,-1.995);
    \end{tikzpicture}
    \caption{The seed \((B,b)\).}
    \label{fig:block-B}
  \end{subfigure}\hfill
  \begin{subfigure}[b]{.49\textwidth}
    \centering
    \begin{tikzpicture}[scale=.78,transform shape]
      \pic (BJ) {blockBJoin};
      \node[attachment vertex] at (BJ-attachmentone) {};
      \node[attachment vertex] at (BJ-attachmenttwo) {};
      \node[font=\small] at (-2.70,-1.72) {\(B_L\)};
      \node[font=\small] at ( 2.70,-1.72) {\(B_R\)};
      \node[below right=1.5pt and 2pt of BJ-attachmentone,font=\scriptsize]
        {\(b_L\)};
      \node[below left=1.3pt and 2pt of BJ-attachmenttwo,font=\scriptsize]
        {\(b_R\)};
      \node[above=3pt of BJ-root,font=\scriptsize] {\(r\)};
    \end{tikzpicture}
    \caption{The symmetric rooted join \((B^{\Join},r)\).}
    \label{fig:block-BJoin}
  \end{subfigure}\hfill
  \begin{subfigure}[b]{.26\textwidth}
    \centering
    \makebox[\linewidth][r]{%
    \begin{tikzpicture}[scale=.94,transform shape]
      \pic[scale=.83,transform shape]
        (JB) at (-1.25,.35) {gadgetB};
      \coordinate (jr) at (.386,.35);
      \coordinate (jw) at (1.15,.35);
      \coordinate (ju) at (.386,1.55);
      \draw[bridge edge] (JB-root)--(jr)--(jw);
      \draw[bridge edge] (jr)--(ju);
      \node[attachment vertex] at (JB-root) {};
      \node[center vertex] at (jr) {};
      \node[attachment vertex] at (jw) {};
      \node[root vertex] at (ju) {};
      \node[font=\scriptsize,below=2pt of jr] {\(r\)};
      \node[font=\scriptsize,right=2pt of jw] {\(w\)};
      \node[font=\scriptsize,above=2pt of ju] {\(u\)};
      \node[font=\small] at (-1.25,-.75)
        {\(B\)};
    \end{tikzpicture}
    }
    \caption{The asymmetric rooted shape \((B^{\ltimes},u)\).}
    \label{fig:graph-Bltimes}
  \end{subfigure}

  \caption{The rooted seed \((B,b)\) and its two derived rooted graphs.}
  \label{fig:B-shapes}
\end{figure}

\begin{lemma}
\label{lem:pendant-root-rule}
Let \((H,x)\) be a rooted graph.  Then
\[
\begin{aligned}
 \gamma_\bullet(H^\uparrow)
 &=
 \bigl(1+\min_i\gamma_i(H),\gamma_0(H),\gamma_1(H)\bigr),\\
 \rho_\bullet(H^\uparrow)
 &=
 \bigl(1+\max\{\rho_2(H),\rho_3(H)\},
       \rho_0(H),\rho_1(H),
       \max\{\rho_2(H),\rho_3(H)\}\bigr).
\end{aligned}
\]
\end{lemma}

\begin{proof}
The new root is selected in state \(0\); the old root is selected in
state \(1\); and the old root is unselected but internally dominated in
state \(2\). The packing formulas follow by adding one to all distances
from the old root and truncating at \(3\).
\end{proof}

The next lemma starts from a flat domination profile and a one-unit
packing asymmetry in state \(0\), and computes how the two derived
operations transform these data.

\begin{lemma}
\label{lem:symbolic-shape-transfers}
Let \((H,x)\) be a rooted graph satisfying
\[
  \gamma_\bullet(H)=(\alpha,\alpha,\alpha),
  \qquad
  \rho_\bullet(H)
  =(\beta+1,\beta,\beta,\beta),
\]
where \(\alpha,\beta\geq0\).
Then
\begin{equation}\label{eq:symmetric-shape-transfer}
  \gamma_\bullet(H^{\Join})
  =(2\alpha+1,2\alpha,2\alpha),
  \qquad
  \rho_\bullet(H^{\Join})
  =(2\beta+1,2\beta+1,2\beta,2\beta),
\end{equation}
and
\begin{equation}\label{eq:asymmetric-shape-transfer}
  \gamma_\bullet(H^{\ltimes})
  =(\alpha+2,\alpha+1,\alpha+1),
  \qquad
  \rho_\bullet(H^{\ltimes})
  =(\beta+1,\beta+1,\beta+1,\beta).
\end{equation}
\end{lemma}

\begin{proof}
Lemma~\ref{lem:pendant-root-rule} first gives
\[
  \gamma_\bullet(H^\uparrow)
  =(\alpha+1,\alpha,\alpha),
  \qquad
  \rho_\bullet(H^\uparrow)
  =(\beta+1,\beta+1,\beta,\beta).
\]
For \(H^{\Join}\), let \(H_L,H_R\) be two vertex-disjoint copies of
\(H\), and let \(p_L,p_R\) be the roots added in
\(H_L^\uparrow,H_R^\uparrow\). The operation \(\odot\) identifies
\(p_L,p_R\) to the common root \(r\).

The two rooted extensions meet only at \(r\). Thus, in
\(H_L^\uparrow\odot H_R^\uparrow\), every boundary domination
candidate or packing restricts to the corresponding local objects on
the two sides, with the membership decision for \(r\) represented in
both restrictions. Conversely, two compatible local witnesses whose
membership decisions agree at \(r\) combine to a global witness.
Distances within either side are unchanged by the identification, and
every path between the two sides passes through \(r\).

In the following state pairs, the first and second entries are the
states of the restrictions to \(H_L^\uparrow\) and
\(H_R^\uparrow\), respectively. For domination, output state \(0\)
corresponds to \((0,0)\), with the common selected root counted twice;
output state \(1\) corresponds to
\((1,1), (1,2),(2,1)\),
and output state \(2\) corresponds to \((2,2)\). These cases determine
the domination profile.
For packing, output state \(0\) again corresponds to \((0,0)\), with
the common selected root counted twice. For an ordered pair \((i,j)\)
of positive local states, compatibility is equivalent to
\(i+j\geq3\), and the output state is \(\min\{i,j\}\). For output
states \(1,2,3\), the maximum values are attained, for example, by
\((1,2)\), \((2,2)\), and \((3,3)\), respectively. These cases give
\eqref{eq:symmetric-shape-transfer}.

For \(H^{\ltimes}\), the same restriction--gluing argument applies
to \(H^\uparrow\odot K_2\), where \(K_2=K_1^\uparrow\) is the rooted
edge whose profiles are given in
\eqref{eq:rooted-edge-profiles}.
The compatible state pairs give
\[
  \gamma_\bullet(H^\uparrow\odot K_2)
  =(\alpha+1,\alpha+1,+\infty),
  \qquad
  \rho_\bullet(H^\uparrow\odot K_2)
  =(\beta+1,\beta+1,\beta,\beta).
\]
One more application of Lemma~\ref{lem:pendant-root-rule} yields
\eqref{eq:asymmetric-shape-transfer}.
\end{proof}

The two derived graphs play complementary roles. The symmetric graph
\(H^{\Join}\) duplicates the seed contribution but retains a boundary
imbalance, whereas \(H^{\ltimes}\) supplies the complementary
correction. They are later used to form the initial rooted graphs,
which are then propagated by connector composition.

We use one common composition for all three connectors in the two
branching constructions.

\begin{definition}
\label{def:connector-composition}
A \emph{rooted connector} is a connected rooted graph \((K,y)\) with
an ordered pair \((z_1,z_2)\) of distinct vertices in
\(V(K)\setminus\{y\}\), called its \emph{child ports}. Let
\((H_1,x_1)\) and \((H_2,x_2)\) be vertex-disjoint rooted graphs,
disjoint from \(K\). Write
\[
  (H_1\ostar_K H_2,y)
\]
for the rooted graph obtained by adding the edges
\(z_1x_1,z_2x_2\). Here, the subscript records the connector, and no
vertices are identified.
\end{definition}

The next lemma shows how the rooted profiles of
\(H_1\ostar_K H_2\) are obtained from the two child profiles and the
possible choices inside \(K\).

\begin{lemma}
\label{lem:edge-attached-connector}
Let \((K,y)\) be a rooted connector with ordered child ports
\((z_1,z_2)\), and let \((H_i,x_i)\), \(i=1,2\), be vertex-disjoint
rooted graphs, disjoint from \(K\). Put
\(G=H_1\ostar_K H_2\).

For \(X\subseteq V(G)\), put \(X_K=X\cap V(K)\) and
\(X_i=X\cap V(H_i)\) for \(i=1,2\).

\begin{enumerate}
\renewcommand{\labelenumi}{(\alph{enumi})}
\item For domination, let \(D\subseteq V(G)\). For \(i=1,2\), suppose
that \(D_i\) is a boundary domination candidate in \((H_i,x_i)\), and
put \(s_i=\sigma_{x_i}^{\mathrm D}(D_i)\).
Then \(D\) is a \(y\)-boundary domination candidate if and only if,
for \(i=1,2\), the port \(z_i\) belongs to \(D_K\) whenever \(s_i=2\),
every vertex of \(V(K)\setminus\{y,z_1,z_2\}\) is dominated by \(D_K\),
and each port \(z_i\) is either dominated by \(D_K\) or has \(s_i=0\).
Moreover,
\[
 \sigma_y^{\mathrm D}(D)=\min\{2,d_K(y,D_K)\},
 \qquad
 |D|\geq |D_K|+\sum_{i=1}^2\gamma_{s_i}(H_i,x_i).
\]

\item For packing, let \(P\subseteq V(G)\). Suppose that \(P_K\) is a
packing in \(K\) and, for \(i=1,2\), that \(P_i\) is a packing in
\(H_i\); put \(s_i=\sigma_{x_i}^{\mathrm P}(P_i)\).
Then \(P\) is a packing in \(G\) if and only if
\[
 s_i+1+d_K(z_i,P_K)\ge3,
 \qquad i=1,2.
\]
Moreover,
\[
 \sigma_y^{\mathrm P}(P)
 =\min\left\{3,d_K(y,P_K),
       \min_{i=1,2}\bigl(d_K(y,z_i)+1+s_i\bigr)\right\},
 \qquad
 |P|\leq |P_K|+\sum_{i=1}^2\rho_{s_i}(H_i,x_i).
\]
\end{enumerate}
The two bounds are attained by extremal child witnesses. Minimizing
in~(a) and maximizing in~(b), separately in each output state, gives
\(\gamma_\bullet(G,y)\) and \(\rho_\bullet(G,y)\).
\end{lemma}

\begin{proof}
Every global object restricts to the connector and the two children,
and compatible restrictions combine, since \(x_iz_i\) is the only
edge joining \(H_i\) to \(K\), for \(i=1,2\).

For~(a),
state \(2\) on either side leaves the corresponding child root
undominated inside its child and therefore forces the corresponding
port into \(D_K\). A port not dominated by \(D_K\) can only be dominated
by a selected child root, corresponding to state \(0\). Since
\(y\notin\{z_1,z_2\}\), no child vertex is within distance \(1\) of
\(y\), so the output state is
\(\sigma_y^{\mathrm D}(D_K)\).

For~(b),
put \(\delta_i=d_{H_i}(x_i,P_i)\), so that
\(s_i=\min\{3,\delta_i\}\). Since \(z_ix_i\) is the unique attachment
edge, compatibility between \(P_K\) and \(P_i\) is equivalent to
\[
 \delta_i+1+d_K(z_i,P_K)\geq3,
\]
which is unchanged when \(\delta_i\) is replaced by \(s_i\). Every
cross-child pair has distance at least
\[
 \delta_1+1+d_K(z_1,z_2)+1+\delta_2\geq3,
\]
since \(z_1\ne z_2\), so no additional condition is needed. The same
attachment structure gives
\[
 d_G(y,P_K)=d_K(y,P_K),
 \qquad
 d_G(y,P_i)=d_K(y,z_i)+1+\delta_i.
\]
Taking the minimum of these three distances and applying
\(\min\{3,\cdot\}\) gives the displayed output formula. Finally, the
three pieces are vertex-disjoint, so
cardinalities add; extremal child witnesses attain the displayed
optima.
\end{proof}

\subsection{The closing lemma}
\label{sec:closing-lemma}

For a rooted graph \((H,x)\), write \(\widehat H\) for the ordinary
unrooted graph obtained from two vertex-disjoint rooted copies
\((H_L,x_L)\) and \((H_R,x_R)\) of \((H,x)\) by adding the edge
\(x_Lx_R\) and then forgetting both root designations. The following lemma converts the
rooted profile of \(H\) into the ordinary domination and packing
numbers of \(\widehat H\).

\begin{lemma}
\label{lem:profile-closure}
Let \((H,x)\) be a rooted graph satisfying
\[
  \gamma_\bullet(H)
  =(\alpha,\alpha,\alpha),
  \qquad
  \rho_\bullet(H)
  =(\beta,\beta,\beta,\beta-1).
\]
Then
\[
  \gamma(\widehat H)=2\alpha,
  \qquad
  \rho(\widehat H)=2\beta.
\]
\end{lemma}

\begin{proof}
Let \(H_L,H_R\) be the two copies used to form \(\widehat H\), with
roots \(x_L,x_R\), respectively. Let \(D\) be a dominating set of
\(\widehat H\), and put
\(D_s=D\cap V(H_s)\) for \(s\in\{L,R\}\). Since no nonroot vertex
of one copy has a neighbor in the other, \(D_s\) is an
\(x_s\)-boundary domination candidate. Hence
\(|D_s|\geq\alpha\) and \(|D|\geq2\alpha\). Conversely, the union of
state-\(0\) domination candidates of size \(\alpha\) in the two
copies dominates \(\widehat H\). Thus \(\gamma(\widehat H)=2\alpha\).

Similarly, if \(P\) is a packing in \(\widehat H\), then
\(P_s=P\cap V(H_s)\) is a packing in \(H_s\) and has one of the four
boundary states. Thus \(|P_s|\leq\beta\) and \(|P|\leq2\beta\).
Conversely, choose a state-\(1\) packing of size \(\beta\) in each
copy. Any selected vertices \(u\in V(H_L)\) and \(v\in V(H_R)\) satisfy
\[
 d_{\widehat H}(u,v)
 =d_{H_L}(u,x_L)+1+d_{H_R}(x_R,v)\geq3,
\]
so the union is a packing of size \(2\beta\). Therefore
\(\rho(\widehat H)=2\beta\).
\end{proof}


\section{\texorpdfstring{The binary branching subcubic family
\(\widehat B_t^\star\)}{The binary branching subcubic family B-hat-star}}
\label{sec:star-family}

We now construct a subcubic family with limiting ratio \(13/6\).

\subsection{\texorpdfstring{The connector \(A\)}{The connector A}}

Let \(A\) be the six-vertex unit displayed in
Figure~\ref{fig:graph-A}; explicitly, it is obtained from the cycle
\(C_5=a_1a_2a_3a_4a_5a_1\) by adding a new vertex \(a_6\) and the
edges \(a_3a_6,a_5a_6\). Set \(a=a_1\), take \(a\) as the root, and
use \((a_4,a_6)\), in this order, as the child ports. Thus \(A\) is a
rooted connector in the sense of
Definition~\ref{def:connector-composition}.

The connector \(A\) is obtained from either of the two order-eight
graphs \(H_1\) and \(H_2\) in Conjecture~\ref{conj:HLR} by deleting
the endpoints of a suitable edge. With the chosen root and child
ports, composing two copies through \(A\) preserves flat domination
profiles and packing profiles with a one-unit deficit in state \(3\).
Thus packing states \(2\) and \(3\) must remain separate in the
following transfer. The transfer uses only the explicit graph \(A\)
and its boundary-profile calculation.

\begin{lemma}
\label{lem:A-balancing}
Let \((H,x)\) be a rooted graph satisfying
\[
  \gamma_\bullet(H)=(\alpha,\alpha,\alpha),
  \qquad
  \rho_\bullet(H)=(\beta,\beta,\beta,\beta-1),
\]
where \(\alpha,\beta\geq0\).  Then
\[
  \gamma_\bullet(H\ostar_A H)
  =(2\alpha+2,2\alpha+2,2\alpha+2),
  \qquad
  \rho_\bullet(H\ostar_A H)
  =(2\beta+1,2\beta+1,2\beta+1,2\beta).
\]
\end{lemma}

\begin{proof}
Apply Lemma~\ref{lem:edge-attached-connector} with connector \(A\).
For a boundary domination candidate \(D\), put
\(D_A=D\cap V(A)\). Each child restriction has size at least
\(\alpha\). The vertices \(a_2,a_3,a_5\) must be dominated by
\(D_A\), and no single vertex of \(A\) dominates all three. Hence
\(|D|\geq2\alpha+2\).

For a packing \(P\), put \(P_A=P\cap V(A)\). Each child restriction
has size at most \(\beta\). Since \(A\) has diameter \(2\),
\(|P_A|\leq1\), and hence \(|P|\leq2\beta+1\). If the output state is
\(3\), then \(P_A=\varnothing\), since every vertex of \(A\) is within
distance \(2\) of \(a\); in this case \(|P|\leq2\beta\). The
compatible rows of
Table~\ref{tab:A-composition-attainment}, together with extremal local
witnesses, attain the corresponding bounds in every output state.
\end{proof}

\begin{table}[tbp]
\centering
\caption{Attainment data for the composition through \(A\).}
\label{tab:A-composition-attainment}
{\renewcommand{\arraystretch}{1.25}
\setlength{\tabcolsep}{3pt}
\begin{tabularx}{\linewidth}{
  |>{\centering\arraybackslash}p{1.15cm}||
   >{\centering\arraybackslash}p{2.75cm}|
   >{\centering\arraybackslash}p{1.70cm}|
   >{\centering\arraybackslash}X||
   >{\centering\arraybackslash}p{2.75cm}|
   >{\centering\arraybackslash}p{1.70cm}|
   >{\centering\arraybackslash}X|
}
\hline
\textbf{State}&
\(\boldsymbol{\gamma}\) \textbf{ child states}&
\(\boldsymbol{D_A}\)&\(\boldsymbol{\gamma_i}\)&
\(\boldsymbol{\rho}\) \textbf{ child states}&
\(\boldsymbol{P_A}\)&\(\boldsymbol{\rho_i}\)\\
\hline\hline
0&\((2,0)\)&\(\{a,a_4\}\)&\(2\alpha+2\)&
  \((0,0)\)&\(\{a\}\)&\(2\beta+1\)\\ \hline
1&\((2,0)\)&\(\{a_2,a_4\}\)&\(2\alpha+2\)&
  \((0,0)\)&\(\{a_2\}\)&\(2\beta+1\)\\ \hline
2&\((2,0)\)&\(\{a_3,a_4\}\)&\(2\alpha+2\)&
  \((2,0)\)&\(\{a_4\}\)&\(2\beta+1\)\\ \hline
3&&&&\((0,0)\)&\(\varnothing\)&\(2\beta\)\\ \hline
\end{tabularx}}
\end{table}

\begin{definition}
\label{def:star-recursion}
For a rooted graph \((H,x)\), define \(\{H_t^\star\}_{t\geq0}\)
recursively by
\[
\begin{aligned}
  H_0^\star
  &\coloneqq
  ((H^{\Join})^\uparrow)^\uparrow\odot H^{\ltimes},\\
  H_{t+1}^\star
  &\coloneqq H_t^\star\ostar_A H_t^\star.
\end{aligned}
\]
The superscript \(\star\) labels the resulting family, while the
subscript in \(\ostar_A\) specifies the connector.
\end{definition}

\begin{proposition}
\label{prop:base-profile-transfer}
Let \((H,x)\) be a rooted graph of order \(n\) satisfying
\[
  \gamma_\bullet(H)=(\alpha,\alpha,\alpha),
  \qquad
\rho_\bullet(H)
  =(\beta+1,\beta,\beta,\beta).
\]
Then
\[
 |V(H_0^\star)|=3n+5,
\]
\[
 \gamma_\bullet(H_0^\star)
 =(3\alpha+2,3\alpha+2,3\alpha+2),
 \qquad
 \rho_\bullet(H_0^\star)
 =(3\beta+2,3\beta+2,3\beta+2,3\beta+1).
\]
\end{proposition}

\begin{table}[tbp]
\centering
\caption{Symbolic orders and profiles used to construct \(H_0^\star\).}
\label{tab:symbolic-base-profiles}
{\renewcommand{\arraystretch}{1.25}
\setlength{\tabcolsep}{3pt}
\begin{tabularx}{\linewidth}{|c||>{\centering\arraybackslash}p{1.5cm}|>{\centering\arraybackslash}X|>{\centering\arraybackslash}X|}
\hline
\textbf{Graph}&\(\boldsymbol{|V|}\)&\(\boldsymbol{\gamma_\bullet}\)&
\(\boldsymbol{\rho_\bullet}\)\\
\hline\hline
\(H\)&\(n\)&\((\alpha,\alpha,\alpha)\)&
\((\beta+1,\beta,\beta,\beta)\)\\ \hline
\(H^{\Join}\)&\(2n+1\)&\((2\alpha+1,2\alpha,2\alpha)\)&
\((2\beta+1,2\beta+1,2\beta,2\beta)\)\\ \hline
\(H^{\ltimes}\)&\(n+3\)&\((\alpha+2,\alpha+1,\alpha+1)\)&
\((\beta+1,\beta+1,\beta+1,\beta)\)\\ \hline
\(H_0^\star\)&\(3n+5\)&\((3\alpha+2,3\alpha+2,3\alpha+2)\)&
\((3\beta+2,3\beta+2,3\beta+2,3\beta+1)\)\\ \hline
\end{tabularx}}
\end{table}

\begin{proof}
The row for \(H\) records the hypothesis. The rows for
\(H^{\Join}\) and \(H^{\ltimes}\) in
Table~\ref{tab:symbolic-base-profiles} follow from
Lemma~\ref{lem:symbolic-shape-transfers}. Applying
Lemma~\ref{lem:pendant-root-rule} twice to \(H^{\Join}\) gives
\[
 \gamma_\bullet(((H^{\Join})^\uparrow)^\uparrow)
 =(2\alpha+1,2\alpha+1,2\alpha+1),
 \qquad
 \rho_\bullet(((H^{\Join})^\uparrow)^\uparrow)
 =(2\beta+2,2\beta+1,2\beta+1,2\beta+1).
\]
The root-identification state matching used in the proof of
Lemma~\ref{lem:symbolic-shape-transfers} then shows that
the graph
\[
 H_0^\star=((H^{\Join})^\uparrow)^\uparrow\odot H^{\ltimes}
\]
has order \(3n+5\) and profiles
\[
 \gamma_\bullet=(3\alpha+2,3\alpha+2,3\alpha+2),
 \qquad
 \rho_\bullet=(3\beta+2,3\beta+2,3\beta+2,3\beta+1).
\]
This gives the last row of the table.
\end{proof}

The seed \(H\) supplies the initial profile data, \(H_0^\star\) puts
them into the forms \((d,d,d)\) and \((p,p,p,p-1)\), and \(A\)
preserves these forms through the recursion. The choice \(H=B\) gives
the concrete family considered next.

\subsection{\texorpdfstring{The completed family \(\widehat B_t^\star\)}{The completed family B-hat-star}}

We now specialize Definition~\ref{def:star-recursion} to the rooted
seed \((B,b)\). Here \(n=9\), and Lemma~\ref{lem:profile-B} shows that
Proposition~\ref{prop:base-profile-transfer} applies with
\(\alpha=3\) and \(\beta=1\), giving
\begin{equation}
\label{eq:B-star-base-data}
 |V(B_0^\star)|=32,\qquad
 \gamma_\bullet(B_0^\star)=(11,11,11),\qquad
\rho_\bullet(B_0^\star)=(5,5,5,4).
\end{equation}

Using the notation in
Section~\ref{sec:closing-lemma}, write \(\widehat B_t^\star\) for the
closure of \(B_t^\star\).
Figure~\ref{fig:star-one-branching} shows the graph
\(\widehat B_1^\star\).

\begin{figure}
  \centering
  \resizebox{.9\textwidth}{!}{%
  \begin{tikzpicture}[
    cycle edge/.append style={line width=.44pt},
    internal edge/.append style={line width=.30pt},
    bridge edge/.append style={line width=.72pt},
    ordinary vertex/.append style={line width=.30pt},
    attachment vertex/.append style={line width=.38pt,double distance=.40pt},
    center vertex/.append style={line width=.38pt,double distance=.40pt},
    root vertex/.append style={line width=.46pt,double distance=.48pt},
    child port/.append style={line width=.46pt,double distance=.48pt},
    tthree child box/.style={
      draw=AuditBlue!40,dashed,rounded corners=5pt,line width=.30pt
    },
    tthree family box/.style={
      draw=AuditGold,dashed,rounded corners=7pt,line width=.32pt
    },
    tthree label/.style={
      font=\fontsize{4.25}{5.0}\selectfont
    },
    tthree family label/.style={
      font=\fontsize{5.2}{6.0}\selectfont
    },
    tthree attachment edge/.style={
      draw=AuditBrown,line width=.44pt
    },
    tthree closing edge/.style={draw=AuditRed,line width=.68pt},
    pics/tthreeBltimes/.style={
      code={
        \coordinate (-root) at (0,0);
        \coordinate (-center) at (0,-1.20);
        \coordinate (-leaf) at (.764,-1.20);
        \pic[scale=.83,transform shape]
          (Bcore) at (-1.636,-1.20) {gadgetB};
        \coordinate (-attachment) at (Bcore-root);
        \draw[bridge edge]
          (-attachment)--(-center)--(-leaf) (-center)--(-root);
        \node[attachment vertex] at (-attachment) {};
        \node[center vertex] at (-center) {};
        \node[attachment vertex] at (-leaf) {};
      }
    },
    pics/tthreeBzero/.style={
      code={
        \coordinate (-root) at (0,0);
        \coordinate (-oldroot) at (0,-.80);
        \pic[scale=.32,yscale=-1,transform shape]
          (J) at (0,-1.2544) {blockBJoin};
        \draw[bridge edge]
          (-root)--(-oldroot)--(J-root);
        \pic[scale=.385,rotate=90,transform shape]
          (L) at (-root) {tthreeBltimes};
        \node[center vertex] at (-root) {};
        \node[ordinary vertex] at (-oldroot) {};
        \node[root vertex,scale=.32,transform shape] at (J-root) {};
      }
    },
    pics/tthreeA/.style={
      code={
        \coordinate (-a1) at (0,1.15);
        \coordinate (-a2) at (-1,.35);
        \coordinate (-a3) at (-.63,-.81);
        \coordinate (-a4) at (.63,-.81);
        \coordinate (-a5) at (1,.35);
        \coordinate (-a6) at (0,.05);
        \draw[cycle edge]
          (-a1)--(-a2)--(-a3)--(-a4)--(-a5)--cycle;
        \draw[internal edge] (-a3)--(-a6)--(-a5);
        \foreach \i in {2,3,5}
          \node[ordinary vertex] at (-a\i) {};
        \node[root vertex] at (-a1) {};
        \foreach \i in {4,6}
          \node[child port] at (-a\i) {};
        \coordinate (-root) at (-a1);
        \coordinate (-portone) at (-a4);
        \coordinate (-porttwo) at (-a6);
      }
    },
    pics/tthreeBone/.style={
      code={
        \pic[scale=.8536585366,rotate=-90,transform shape] (A) {tthreeA};
        \coordinate (-lowerroot) at (-1.72,-1.3487805);
        \coordinate (-upperroot) at (-1.72, 1.3487805);
        \pic[scale=.8536585366,transform shape]
          (Clo) at (-lowerroot) {tthreeBzero};
        \pic[scale=.8536585366,yscale=-1,transform shape]
          (Cup) at (-upperroot) {tthreeBzero};
        \begin{scope}[on background layer]
          \draw[tthree attachment edge]
            (A-portone) .. controls (-1.22,-.54) and (-1.72,-.92)
            .. (-lowerroot);
          \draw[tthree attachment edge]
            (A-porttwo) .. controls (-.55,.08) and (-1.72,.70)
            .. (-upperroot);
        \end{scope}
        \coordinate (-root) at (A-root);
        \foreach \v in {A-portone,A-porttwo}
          \node[child port,scale=.8536585366,transform shape] at (\v) {};
        \foreach \v in {lowerroot,upperroot}
          \node[center vertex,scale=.8536585366,transform shape] at (-\v) {};
        \node[root vertex,scale=.8536585366,transform shape] at (-root) {};
      }
    }
  ]
    \begin{scope}[on background layer]
      \draw[tthree family box]
        (-4.42,-2.82) rectangle (-.28,2.82);
      \draw[tthree family box]
        (.28,-2.82) rectangle (4.42,2.82);
    \end{scope}
    \draw[tthree child box]
      (-4.25,.58) rectangle (-1.75,2.38);
    \draw[tthree child box]
      (-4.25,-2.38) rectangle (-1.75,-.58);
    \draw[tthree child box]
      (1.75,.58) rectangle (4.25,2.38);
    \draw[tthree child box]
      (1.75,-2.38) rectangle (4.25,-.58);

    \pic[scale=.82,transform shape]
      (BL) at (-1.35,0) {tthreeBone};
    \pic[scale=.82,xscale=-1,transform shape]
      (BR) at (1.35,0) {tthreeBone};

    \draw[tthree closing edge] (BL-root)--(BR-root);
    \foreach \v in {BL-root,BR-root}
      \node[root vertex,scale=.70,transform shape] at (\v) {};

    \node[tthree label] at (-1.35,.88) {\(A\)};
    \node[tthree label] at (1.35,.88) {\(A\)};
    \node[tthree label] at (-3.00,2.57) {\(B_0^\star\)};
    \node[tthree label] at (-3.00,-2.57) {\(B_0^\star\)};
    \node[tthree label] at (3.00,2.57) {\(B_0^\star\)};
    \node[tthree label] at (3.00,-2.57) {\(B_0^\star\)};
    \node[tthree family label] at (-2.45,3.08) {\(B_1^\star\)};
    \node[tthree family label] at (2.45,3.08) {\(B_1^\star\)};
  \end{tikzpicture}%
  }
  \caption{The graph \(\widehat B_1^\star\). Each blue dashed box
  isolates a rooted copy of \(B_0^\star\), and each gold dashed box
  isolates a rooted copy of \(B_1^\star\).}
  \label{fig:star-one-branching}
\end{figure}

\begin{lemma}
\label{lem:branching-structure}
For every \(t\geq0\), the graph \(B_t^\star\) is connected and
subcubic, and its root has degree \(2\). Consequently,
\(\widehat B_t^\star\) is connected, subcubic, and noncubic.
\end{lemma}

\begin{proof}
The definition of \(B_0^\star\) shows that it is connected and
subcubic and that its root has degree \(2\). Suppose the assertion
holds for \(B_t^\star\). Then
\(B_{t+1}^\star=B_t^\star\ostar_A B_t^\star\) is connected, and its
two attachment edges raise the child roots and the ports
\(a_4,a_6\) from degree \(2\) to \(3\). The new root \(a\) has degree
\(2\), proving the first assertion by induction.

The closing edge raises both roots to degree \(3\) and leaves all other
degrees unchanged, so \(\widehat B_t^\star\) is connected and
subcubic. For \(t=0\), the
root introduced by the first pendant-root extension is a nonroot
degree-two vertex; for \(t\geq1\), so is \(a_2\) in the outermost
\(A\). Thus \(\widehat B_t^\star\) is noncubic.
\end{proof}

\begin{theorem}
\label{thm:star-parameters}
For every \(t\geq0\), the graph \(\widehat B_t^\star\) is connected,
subcubic, and noncubic, and satisfies
\[
 |V(\widehat B_t^\star)|=76\cdot2^t-12,
 \qquad
 \gamma(\widehat B_t^\star)=26\cdot2^t-4,
 \qquad
 \rho(\widehat B_t^\star)=12\cdot2^t-2.
\]
Consequently,
\[
 \gamma(\widehat B_t^\star)
 -2\rho(\widehat B_t^\star)=2^{t+1},
\]
and
\[
 \frac{\gamma(\widehat B_t^\star)}{\rho(\widehat B_t^\star)}
 =\frac{13}{6}+\frac{1}{36\cdot2^t-6}
 \,\xrightarrow[t\to\infty]{}\,\frac{13}{6}.
\]
\end{theorem}

\begin{proof}
The structural assertions follow from
Lemma~\ref{lem:branching-structure}.
We first carry out the calculation for an arbitrary rooted seed
\((H,x)\) of order \(n\) satisfying the hypotheses of
Proposition~\ref{prop:base-profile-transfer}, with profile parameters
\(\alpha\) and \(\beta\). By Proposition~\ref{prop:base-profile-transfer} and
Lemma~\ref{lem:A-balancing}, there are sequences \((d_t)\) and
\((p_t)\) such that
\[
 \gamma_\bullet(H_t^\star)=(d_t,d_t,d_t),
 \qquad
 \rho_\bullet(H_t^\star)=(p_t,p_t,p_t,p_t-1).
\]
Set \(v_t=|V(H_t^\star)|\). Then
\[
\begin{alignedat}{3}
 d_0&=3\alpha+2,\qquad&
 p_0&=3\beta+2,\qquad&
 v_0&=3n+5,\\
 d_{t+1}&=2d_t+2,\qquad&
 p_{t+1}&=2p_t+1,\qquad&
 v_{t+1}&=2v_t+6.
\end{alignedat}
\]
Hence
\[
\begin{alignedat}{3}
 d_t&=(3\alpha+4)2^t-2,\qquad&
 p_t&=(3\beta+3)2^t-1,\qquad&
 v_t&=(3n+11)2^t-6.
\end{alignedat}
\]
Write \(\widehat H_t^\star\) for the closure of \(H_t^\star\).
The closing construction and Lemma~\ref{lem:profile-closure} give
\[
 |V(\widehat H_t^\star)|=2v_t,
 \qquad
 \gamma(\widehat H_t^\star)=2d_t,
 \qquad
 \rho(\widehat H_t^\star)=2p_t.
\]
Consequently,
\[
 \frac{\gamma(\widehat H_t^\star)}
      {\rho(\widehat H_t^\star)}
 \,\xrightarrow[t\to\infty]{}\,
 \Lambda(H)\coloneqq\frac{3\alpha+4}{3\beta+3}.
\]
This calculation uses only the profile hypotheses. If
\(H\) is also connected and subcubic with \(\deg_H(x)\leq2\), the
argument of Lemma~\ref{lem:branching-structure} shows that its
closures are connected, noncubic, and subcubic.

For \(H=B\), we have \(n=9\), while Lemma~\ref{lem:profile-B} gives
\(\alpha=3\) and \(\beta=1\). Substitution gives the three parameter
formulas; the gap identity and ratio formula follow by direct
simplification.
\end{proof}

For any seed satisfying the profile and structural conditions in the
proof of Theorem~\ref{thm:star-parameters}, integrality of \(\alpha\)
and \(\beta\) shows that it improves the limit attained by \(B\)
precisely when
\[
 \Lambda(H)>\Lambda(B)=\frac{13}{6}
 \quad\Longleftrightarrow\quad
 6\alpha-13\beta>5
 \quad\Longleftrightarrow\quad
 6\alpha-13\beta\geq6.
\]
This is a search criterion within the present mechanism; it neither
asserts that such a seed exists nor that \(B\) is globally optimal.

\section{\texorpdfstring{The binary branching cubic family
\(\widehat B_t^\bullet\)}{The binary branching cubic family B-hat-bullet}}
\label{sec:bullet-family}

We now construct a cubic family with limiting ratio \(17/8\), using
the common composition from
Definition~\ref{def:connector-composition} with two new connectors.

\subsection{\texorpdfstring{The connectors \(S\) and \(R\)}{The connectors S and R}}

\begin{definition}
Let the \emph{seed connector} be \(S=B-b_4b_7\), rooted at \(b\) with
ordered child ports \((b_4,b_7)\). The \emph{recursive connector}
\(R\) is obtained from \(A\) by subdividing \(a_3a_4\) with a new
vertex \(a_7\) and adding the edge \(a_2a_7\); take \(a\) as its root
and \((a_4,a_6)\) as its ordered child ports.
\end{definition}

In both \(S\) and \(R\), the root and two child ports have degree two,
and every other vertex has degree three; see
Figure~\ref{fig:bullet-connectors}.

\begin{figure}
\centering
\begin{minipage}[t]{.44\textwidth}\centering
\begin{tikzpicture}[
  scale=.82,
  cycle edge/.append style={line width=.54pt},
  internal edge/.append style={line width=.38pt},
  connector hollow/.style={
    circle,draw=AuditBlue,fill=white,
    minimum size=3.2pt,inner sep=0pt,line width=.36pt
  },
  root vertex/.append style={
    minimum size=5.6pt,line width=.58pt,double distance=.60pt
  },
  child port/.append style={
    minimum size=5.6pt,line width=.58pt,double distance=.60pt
  }
]
\foreach \i/\a in {1/90,2/130,3/170,4/-150,5/-110,6/-70,7/-30,8/10,9/50}
  \coordinate(cb\i) at (\a:1.18);
\draw[cycle edge]
  (cb1)--(cb2)--(cb3)--(cb4)--(cb5)--(cb6)--(cb7)--(cb8)--(cb9)--cycle;
\foreach \a/\b in {2/6,3/8,5/9}
  \draw[internal edge] (cb\a)--(cb\b);
\foreach \i in {2,3,5,6,8,9}
  \node[connector hollow] at (cb\i) {};
\node[root vertex] at (cb1) {};
\foreach \i in {4,7}
  \node[child port] at (cb\i) {};
\node[font=\scriptsize,above=2pt] at (cb1) {\(b_1=b\)};
\node[font=\scriptsize,above left=1pt and 1pt] at (cb2) {\(b_2\)};
\node[font=\scriptsize,left=2pt] at (cb3) {\(b_3\)};
\node[font=\scriptsize,below left=1pt and 1pt] at (cb4) {\(b_4\)};
\node[font=\scriptsize,below left=1pt and 1pt] at (cb5) {\(b_5\)};
\node[font=\scriptsize,below right=1pt and 1pt] at (cb6) {\(b_6\)};
\node[font=\scriptsize,below right=1pt and 1pt] at (cb7) {\(b_7\)};
\node[font=\scriptsize,right=2pt] at (cb8) {\(b_8\)};
\node[font=\scriptsize,above right=1pt and 1pt] at (cb9) {\(b_9\)};
\end{tikzpicture}
\par\smallskip
\(S\)
\end{minipage}\hfill
\begin{minipage}[t]{.44\textwidth}\centering
\begin{tikzpicture}[
  scale=.95,
  cycle edge/.append style={line width=.54pt},
  internal edge/.append style={line width=.38pt},
  ordinary vertex/.append style={minimum size=3.2pt,line width=.38pt},
  root vertex/.append style={
    minimum size=5.6pt,line width=.58pt,double distance=.60pt
  },
  child port/.append style={
    minimum size=5.6pt,line width=.58pt,double distance=.60pt
  }
]
\coordinate (ra1) at (0,1.15);
\coordinate (ra2) at (-1.00,.35);
\coordinate (ra3) at (-.63,-.81);
\coordinate (ra4) at ( .63,-.81);
\coordinate (ra5) at (1.00,.35);
\coordinate (ra6) at (0,.05);
\coordinate (ra7) at ($(ra3)!.50!(ra4)$);
\draw[cycle edge]
  (ra1)--(ra2)--(ra3)--(ra7)--(ra4)--(ra5)--cycle;
\draw[internal edge](ra3)--(ra6)--(ra5);
\draw[internal edge](ra2)--(ra7);
\foreach \i in {2,3,5,7}\node[ordinary vertex]at(ra\i){};
\node[root vertex]at(ra1){};
\foreach \i in {4,6}\node[child port]at(ra\i){};
\node[font=\scriptsize,above=2pt]at(ra1){\(a_1=a\)};
\node[font=\scriptsize,left=2pt]at(ra2){\(a_2\)};
\node[font=\scriptsize,below left=1pt]at(ra3){\(a_3\)};
\node[font=\scriptsize,below=2pt]at(ra7){\(a_7\)};
\node[font=\scriptsize,below right=1pt]at(ra4){\(a_4\)};
\node[font=\scriptsize,right=2pt]at(ra5){\(a_5\)};
\node[font=\scriptsize,left=2pt]at(ra6){\(a_6\)};
\end{tikzpicture}
\par\smallskip
\(R\)
\end{minipage}
\caption{The seed connector \(S\) and the recursive connector \(R\).
Roots and child ports are highlighted.}
\label{fig:bullet-connectors}
\end{figure}
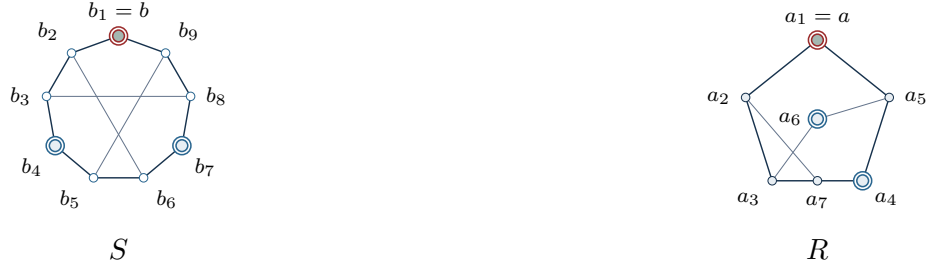

\begin{definition}
\label{def:connector-recursion}
For a rooted graph \((H,x)\), define \(\{H_t^\bullet\}_{t\geq0}\)
recursively by
\[
\begin{aligned}
 H_0^\bullet
 &\coloneqq H^{\Join}\ostar_S H^{\Join},\\
 H_{t+1}^\bullet
 &\coloneqq H_t^\bullet\ostar_R H_t^\bullet.
\end{aligned}
\]
\end{definition}

The next two lemmas show that \(S\) initializes the recursive profile
class and \(R\) preserves it.

\begin{lemma}\label{lem:connector-base-profile}
Let \((H,x)\) be a rooted graph of order \(n\) satisfying
\[
 \gamma_\bullet(H)=(\alpha,\alpha,\alpha),\qquad
 \rho_\bullet(H)=(\beta+1,\beta,\beta,\beta).
\]
Then
\[
 |V(H_0^\bullet)|=4n+11,
\]
\[
 \gamma_\bullet(H_0^\bullet)
 =(4\alpha+3,4\alpha+3,4\alpha+3),
 \qquad
 \rho_\bullet(H_0^\bullet)
 =(4\beta+3,4\beta+3,4\beta+3,4\beta+2).
\]
\end{lemma}

\begin{proof}
The child profiles are given by
\eqref{eq:symmetric-shape-transfer}. Also
\(|V(H^{\Join})|=2n+1\), so the definition gives
\[
|V(H_0^\bullet)|=9+2(2n+1)=4n+11.
\]

Apply Lemma~\ref{lem:edge-attached-connector} with connector \(S\).
For a boundary domination candidate \(D\), put
\(D_S=D\cap V(S)\). Each copy of \(H^{\Join}\) contributes at least
\(2\alpha\), with one additional vertex in state \(0\). Thus
\(|D|\) is at least \(4\alpha+|D_S|\) plus the number of copies in
state \(0\). The six internal vertices
\(b_2,b_3,b_5,b_6,b_8,b_9\) must be dominated by \(D_S\), so
\(|D_S|\geq2\). If \(|D_S|=2\), the only pairs that dominate all six are
\[
 \{b_2,b_9\},\qquad \{b_3,b_5\},\qquad \{b_6,b_8\}.
\]
The first pair forces both copies into state \(0\), and each of the
other two pairs forces at least one copy into state \(0\). Hence
\(|D|\geq4\alpha+3\).

For a packing \(P\), put \(P_S=P\cap V(S)\). Each copy of
\(H^{\Join}\) contributes at most \(2\beta\), with one additional
vertex in state \(0\) or \(1\). If \(|P_S|\leq1\), these additional
vertices give \(|P|\leq4\beta+3\). If \(|P_S|\geq2\), then
\(P_S\subseteq\{b,b_4,b_7\}\) and at least \(|P_S|-1\) child ports are
selected. Each selected port excludes states \(0\) and \(1\) in the
corresponding copy, so again \(|P|\leq4\beta+3\). In output state
\(3\), \(P_S\subseteq\{b_4,b_7\}\), and the same argument gives
\(|P|\leq4\beta+2\).
The witnesses in Table~\ref{tab:S-base-attainment} attain all these
bounds.
\end{proof}

\begin{table}[tbp]
\centering
\caption{Attainment data for the composition through \(S\).}
\label{tab:S-base-attainment}
{\renewcommand{\arraystretch}{1.25}
\setlength{\tabcolsep}{3pt}
\begin{tabularx}{\linewidth}{
  |>{\centering\arraybackslash}p{1.15cm}||
   >{\centering\arraybackslash}p{2.75cm}|
   >{\centering\arraybackslash}p{2.00cm}|
   >{\centering\arraybackslash}X||
   >{\centering\arraybackslash}p{2.75cm}|
   >{\centering\arraybackslash}p{2.00cm}|
   >{\centering\arraybackslash}X|
}
\hline
\textbf{State}&
\(\boldsymbol{\gamma}\) \textbf{ child states}&
\(\boldsymbol{D_S}\)&\(\boldsymbol{\gamma_i}\)&
\(\boldsymbol{\rho}\) \textbf{ child states}&
\(\boldsymbol{P_S}\)&\(\boldsymbol{\rho_i}\)\\
\hline\hline
0&\((1,1)\)&\(\{b,b_4,b_7\}\)&\(4\alpha+3\)&
  \((0,0)\)&\(\{b\}\)&\(4\beta+3\)\\ \hline
1&\((1,1)\)&\(\{b_2,b_4,b_8\}\)&\(4\alpha+3\)&
  \((1,1)\)&\(\{b_2\}\)&\(4\beta+3\)\\ \hline
2&\((0,1)\)&\(\{b_6,b_8\}\)&\(4\alpha+3\)&
  \((1,1)\)&\(\{b_3\}\)&\(4\beta+3\)\\ \hline
3&&&&\((0,0)\)&\(\varnothing\)&\(4\beta+2\)\\ \hline
\end{tabularx}}
\end{table}

\begin{lemma}
\label{lem:connector-recursive-transfer}
If a rooted graph \((H,x)\) has
\[
 \gamma_\bullet(H)=(\alpha,\alpha,\alpha),\qquad
 \rho_\bullet(H)=(\beta,\beta,\beta,\beta-1),
\]
then
\[
 \gamma_\bullet(H\ostar_R H)=(2\alpha+2,2\alpha+2,2\alpha+2),
 \qquad
 \rho_\bullet(H\ostar_R H)
 =(2\beta+1,2\beta+1,2\beta+1,2\beta).
\]
\end{lemma}

\begin{table}[t]
\centering
\caption{Attainment data for the composition through \(R\).}
\label{tab:R-recursive-attainment}
{\renewcommand{\arraystretch}{1.25}
\setlength{\tabcolsep}{3pt}
\begin{tabularx}{\linewidth}{
  |>{\centering\arraybackslash}p{1.15cm}||
   >{\centering\arraybackslash}p{2.75cm}|
   >{\centering\arraybackslash}p{2.00cm}|
   >{\centering\arraybackslash}X||
   >{\centering\arraybackslash}p{2.75cm}|
   >{\centering\arraybackslash}p{2.00cm}|
   >{\centering\arraybackslash}X|
}
\hline
\textbf{State}&
\(\boldsymbol{\gamma}\) \textbf{ child states}&
\(\boldsymbol{D_R}\)&\(\boldsymbol{\gamma_i}\)&
\(\boldsymbol{\rho}\) \textbf{ child states}&
\(\boldsymbol{P_R}\)&\(\boldsymbol{\rho_i}\)\\
\hline\hline
0&\((0,0)\)&\(\{a,a_2\}\)&\(2\alpha+2\)&
  \((0,0)\)&\(\{a\}\)&\(2\beta+1\)\\ \hline
1&\((0,0)\)&\(\{a_2,a_4\}\)&\(2\alpha+2\)&
  \((0,0)\)&\(\{a_2\}\)&\(2\beta+1\)\\ \hline
2&\((0,0)\)&\(\{a_4,a_7\}\)&\(2\alpha+2\)&
  \((2,0)\)&\(\{a_4\}\)&\(2\beta+1\)\\ \hline
3&&&&\((0,0)\)&\(\varnothing\)&\(2\beta\)\\ \hline
\end{tabularx}}
\end{table}

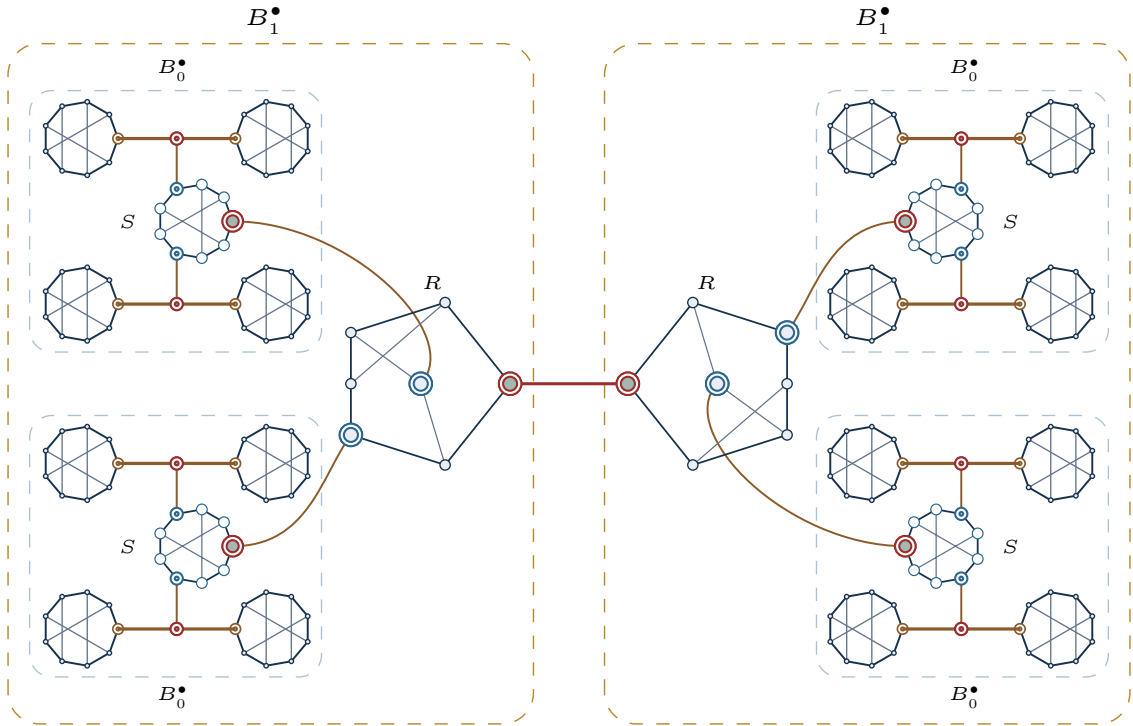
\begin{figure}[b]
\centering
\resizebox{.9\textwidth}{!}{%
\begin{tikzpicture}[
  cycle edge/.append style={line width=.44pt},
  internal edge/.append style={line width=.30pt},
  bridge edge/.append style={line width=.72pt},
  ordinary vertex/.append style={line width=.30pt},
  attachment vertex/.append style={line width=.38pt,double distance=.40pt},
  center vertex/.append style={line width=.38pt,double distance=.40pt},
  root vertex/.append style={line width=.46pt,double distance=.48pt},
  child port/.append style={line width=.46pt,double distance=.48pt},
  wone connector cycle/.style={
    draw=AuditNavy,line width=.40pt
  },
  wone connector chord/.style={
    draw=AuditGray,line width=.28pt
  },
  wone connector hollow/.style={
    circle,draw=AuditBlue,fill=white,
    minimum size=2.4pt,inner sep=0pt,line width=.27pt
  },
  wone closing edge/.style={
    draw=AuditRed,line width=.68pt
  },
  wone attachment edge/.style={
    draw=AuditBrown,line width=.44pt
  },
  wone family box/.style={
    draw=AuditGold,dashed,rounded corners=7pt,line width=.32pt
  },
  wone label/.style={
    font=\fontsize{4.25}{5.0}\selectfont
  },
  pics/woneBzero/.style={
    code={
      \coordinate (-root) at (0,0);
      \coordinate (-Scenter) at (-.294,0);
      \foreach \i/\a in {1/0,2/40,3/80,5/160,6/200,8/280,9/320}
        \coordinate (-b\i) at ($(-Scenter)+(\a:.294)$);
      %
      \pic[scale=.28,transform shape]
        (upper) at (-.441,.65221) {blockBJoin};
      \pic[scale=.28,rotate=180,transform shape]
        (lower) at (-.441,-.65221) {blockBJoin};
      \coordinate (-b4) at (-.441,.25461);
      \coordinate (-b7) at (-.441,-.25461);
      \draw[wone connector cycle]
        (-b1)--(-b2)--(-b3)--(-b4)--(-b5)--
        (-b6)--(-b7)--(-b8)--(-b9)--cycle;
      \draw[wone connector chord]
        (-b2)--(-b6) (-b3)--(-b8) (-b5)--(-b9);
      \draw[wone attachment edge]
        (-b4)--(upper-root) (-b7)--(lower-root);
      \foreach \i in {2,3,5,6,8,9}
        \node[wone connector hollow] at (-b\i) {};
      \foreach \v in {b4,b7}
        \node[child port,scale=.28,transform shape] at (-\v) {};
      \foreach \v in {upper-root,lower-root}
        \node[root vertex,scale=.28,transform shape] at (\v) {};
      \node[root vertex,scale=.57,transform shape] at (-root) {};
    }
  },
  pics/woneR/.style={
    code={
      \coordinate (-a1) at (0,1.15);
      \coordinate (-a2) at (-1.00,.35);
      \coordinate (-a3) at (-.63,-.81);
      \coordinate (-a4) at ( .63,-.81);
      \coordinate (-a5) at (1.00,.35);
      \coordinate (-a6) at (0,.05);
      \coordinate (-a7) at ($(-a3)!.50!(-a4)$);
      \draw[cycle edge,line width=.40pt]
        (-a1)--(-a2)--(-a3)--(-a7)--(-a4)--(-a5)--cycle;
      \draw[internal edge,line width=.28pt] (-a3)--(-a6)--(-a5);
      \draw[internal edge,line width=.28pt] (-a2)--(-a7);
      \foreach \i in {2,3,5,7}
        \node[ordinary vertex] at (-a\i) {};
      \node[root vertex] at (-a1) {};
      \foreach \i in {4,6}
        \node[child port] at (-a\i) {};
      \coordinate (-root) at (-a1);
      \coordinate (-portone) at (-a4);
      \coordinate (-porttwo) at (-a6);
    }
  }
]
  \begin{scope}[on background layer]
    \draw[wone family box]
      (-4.42,-2.68) rectangle (-.28,2.68);
    \draw[wone family box]
      (.28,-2.68) rectangle (4.42,2.68);
  \end{scope}

  \draw[
    draw=AuditBlue!40,dashed,rounded corners=5pt,line width=.30pt
  ] (-4.25,.25) rectangle (-1.95,2.31);
  \draw[
    draw=AuditBlue!40,dashed,rounded corners=5pt,line width=.30pt
  ] (-4.25,-2.31) rectangle (-1.95,-.25);
  \draw[
    draw=AuditBlue!40,dashed,rounded corners=5pt,line width=.30pt
  ] (1.95,.25) rectangle (4.25,2.31);
  \draw[
    draw=AuditBlue!40,dashed,rounded corners=5pt,line width=.30pt
  ] (1.95,-2.31) rectangle (4.25,-.25);

  \pic (BLU) at (-2.65,1.28) {woneBzero};
  \pic (BLD) at (-2.65,-1.28) {woneBzero};
  \pic[rotate=180,transform shape] (BRU) at (2.65,1.28) {woneBzero};
  \pic[rotate=180,transform shape] (BRD) at (2.65,-1.28) {woneBzero};

  \pic[scale=.64,rotate=-90,transform shape] (RL) at (-1.20,0) {woneR};
  \pic[scale=.64,rotate=90,transform shape] (RR) at (1.20,0) {woneR};

  \begin{scope}[on background layer]
    \draw[wone attachment edge]
      (RL-portone) .. controls (-1.95,-.72) and (-2.05,-1.28) .. (BLD-root);
    \draw[wone attachment edge]
      (RL-porttwo) .. controls (-.82,.38) and (-1.72,1.28) .. (BLU-root);
    \draw[wone attachment edge]
      (RR-portone) .. controls (1.95,.72) and (2.05,1.28) .. (BRU-root);
    \draw[wone attachment edge]
      (RR-porttwo) .. controls (.82,-.38) and (1.72,-1.28) .. (BRD-root);
  \end{scope}
  \draw[wone closing edge] (RL-root)--(RR-root);

  \foreach \v in {RL-portone,RL-porttwo,RR-portone,RR-porttwo}
    \node[child port,scale=.64,transform shape] at (\v) {};
  \foreach \v in {BLU-root,BLD-root,BRU-root,BRD-root}
    \node[root vertex,scale=.57,transform shape] at (\v) {};
  \foreach \v in {RL-root,RR-root}
    \node[root vertex,scale=.64,transform shape] at (\v) {};

  \node[wone label] at (-1.08,.80) {\(R\)};
  \node[wone label] at ( 1.08,.80) {\(R\)};
  \node[wone label,anchor=east] at ($(BLU-Scenter)+(-.33,0)$) {\(S\)};
  \node[wone label,anchor=east] at ($(BLD-Scenter)+(-.33,0)$) {\(S\)};
  \node[wone label,anchor=west] at ($(BRU-Scenter)+(.33,0)$) {\(S\)};
  \node[wone label,anchor=west] at ($(BRD-Scenter)+(.33,0)$) {\(S\)};
  \node[wone label] at (-3.12,2.47) {\(B_0^\bullet\)};
  \node[wone label] at (-3.12,-2.47) {\(B_0^\bullet\)};
  \node[wone label] at ( 3.12,2.47) {\(B_0^\bullet\)};
  \node[wone label] at ( 3.12,-2.47) {\(B_0^\bullet\)};
  \node[font=\fontsize{5.2}{6.0}\selectfont] at (-2.40,2.86)
    {\(B_1^\bullet\)};
  \node[font=\fontsize{5.2}{6.0}\selectfont] at ( 2.40,2.86)
    {\(B_1^\bullet\)};
\end{tikzpicture}%
}
\caption{The graph \(\widehat B_1^\bullet\). Each blue dashed box
isolates a rooted copy of \(B_0^\bullet\), and each gold dashed box
isolates a rooted copy of \(B_1^\bullet\).}
\label{fig:bullet-one-recursive}
\end{figure}

\begin{proof}
Apply Lemma~\ref{lem:edge-attached-connector} with connector \(R\).
For a boundary domination candidate \(D\), put
\(D_R=D\cap V(R)\). Each child restriction has size at least
\(\alpha\). The internal vertices \(a_2,a_3,a_5,a_7\) cannot be
dominated from a child root, and no single vertex of \(R\) dominates
all four. Hence \(|D|\geq2\alpha+2\).

For a packing \(P\), put \(P_R=P\cap V(R)\). Each child restriction
has size at most \(\beta\). Since \(R\) has diameter \(2\),
\(|P_R|\leq1\), and hence \(|P|\leq2\beta+1\). If the output state is
\(3\), then \(P_R=\varnothing\), since every vertex of \(R\) is within
distance \(2\) of \(a\); in this case \(|P|\leq2\beta\).
The witnesses in Table~\ref{tab:R-recursive-attainment} attain the
corresponding bounds in every output state.
\end{proof}

\subsection{\texorpdfstring{The completed family \(\widehat B_t^\bullet\)}{The completed family B-hat-bullet}}

We now specialize Definition~\ref{def:connector-recursion} to the
rooted seed \((B,b)\). Here \(n=9\), and Lemma~\ref{lem:profile-B}
shows that Lemma~\ref{lem:connector-base-profile} applies with
\(\alpha=3\) and \(\beta=1\). Hence
\begin{equation}
\label{eq:B-bullet-base-data}
 |V(B_0^\bullet)|=47,\qquad
 \gamma_\bullet(B_0^\bullet)=(15,15,15),\qquad
 \rho_\bullet(B_0^\bullet)=(7,7,7,6).
\end{equation}

As in the subcubic construction, write \(\widehat B_t^\bullet\) for the
closure of \(B_t^\bullet\).
The graph \(\widehat B_1^\bullet\) is shown in
Figure~\ref{fig:bullet-one-recursive}.

\begin{lemma}\label{lem:bullet-structure}
For every \(t\geq0\), the graph \(B_t^\bullet\) is connected, its root
has degree \(2\), and every other vertex has degree \(3\). Consequently,
\(\widehat B_t^\bullet\) is connected and cubic.
\end{lemma}

\begin{proof}
The definition of \(B_0^\bullet\) shows that it is connected, its root
has degree \(2\), and every other vertex has degree \(3\). Suppose the
assertion holds for \(B_t^\bullet\). Then
\(B_{t+1}^\bullet=B_t^\bullet\ostar_R B_t^\bullet\) is connected, and
its two attachment edges raise the child roots and the ports
\(a_4,a_6\) from degree \(2\) to \(3\). The new root \(a\) has degree
\(2\), proving the assertion by induction.

The closing edge raises both roots to degree \(3\) and leaves all other
degrees unchanged. Hence \(\widehat B_t^\bullet\) is connected and
cubic.
\end{proof}

\begin{theorem}\label{thm:bullet-parameters}
For every \(t\geq0\), the graph \(\widehat B_t^\bullet\) is connected
and cubic, and satisfies
\[
 |V(\widehat B_t^\bullet)|=108\cdot2^t-14,
 \qquad
 \gamma(\widehat B_t^\bullet)=34\cdot2^t-4,
 \qquad
 \rho(\widehat B_t^\bullet)=16\cdot2^t-2.
\]
Consequently,
\[
 \gamma(\widehat B_t^\bullet)
 -2\rho(\widehat B_t^\bullet)=2^{t+1},
\]
and
\[
 \frac{\gamma(\widehat B_t^\bullet)}{\rho(\widehat B_t^\bullet)}
 =\frac{17}{8}
 +\frac{1}{64\cdot2^t-8}
 \,\xrightarrow[t\to\infty]{}\,\frac{17}{8}.
\]
\end{theorem}

\begin{proof}
The structural assertions follow from
Lemma~\ref{lem:bullet-structure}.
Equation~\eqref{eq:B-bullet-base-data} and
Lemma~\ref{lem:connector-recursive-transfer} show that there are
sequences \((d_t)\) and \((p_t)\) such that
\[
 \gamma_\bullet(B_t^\bullet)=(d_t,d_t,d_t),\qquad
 \rho_\bullet(B_t^\bullet)
 =(p_t,p_t,p_t,p_t-1).
\]
Set \(v_t=|V(B_t^\bullet)|\). Then
\[
\begin{alignedat}{3}
 d_0&=15,\qquad&
 p_0&=7,\qquad&
 v_0&=47,\\
 d_{t+1}&=2d_t+2,\qquad&
 p_{t+1}&=2p_t+1,\qquad&
 v_{t+1}&=2v_t+7.
\end{alignedat}
\]
Hence
\[
\begin{alignedat}{3}
 d_t&=17\cdot2^t-2,\qquad&
 p_t&=8\cdot2^t-1,\qquad&
 v_t&=54\cdot2^t-7.
\end{alignedat}
\]
The closing construction and Lemma~\ref{lem:profile-closure} give
\[
 |V(\widehat B_t^\bullet)|=2v_t,
 \qquad
 \gamma(\widehat B_t^\bullet)=2d_t,
 \qquad
 \rho(\widehat B_t^\bullet)=2p_t.
\]
Substitution gives the stated formulas; the gap identity and ratio
formula follow by direct simplification.
\end{proof}

\section{Concluding remarks}
\label{sec:consequences}
The two branching families give
\[
 c_{\mathrm{cub}}\geq\frac{17}{8},\qquad
 c_{\mathrm{sub}}\geq\frac{13}{6}.
\]
For a connected cubic graph \(G\) of order \(n=|V(G)|>8\), the
domination bound of Kostochka and Stocker
\cite{KostochkaStocker2009} and the packing bound of Goddard and
Henning \cite{GoddardHenning2024} give, respectively,
\[
 \gamma(G)\leq\frac5{14}n,\qquad
 \rho(G)>\frac{17}{132}(n-3).
\]
Combining the two inequalities gives
\[
 n<\frac{132}{17}\rho(G)+3
 \quad\text{and hence}\quad
 \gamma(G)<\frac{330}{119}\rho(G)+\frac{15}{14}.
\]
The finitely many connected cubic graphs of order at most \(8\) do not
affect the limsup. Therefore
\[
 \frac{17}{8}\leq c_{\mathrm{cub}}\leq\frac{330}{119}.
\]

The exact family formulas establish the asymptotic lower bounds above
and, as a consequence, give strict counterexamples to
Conjecture~\ref{conj:HLR}.
Motivated by these bounds, we propose three revised statements. 
The leading coefficients in parts~(2) and~(3) are forced by our
branching families, while the additive constants make the three
original equality graphs \(H_1,H_2,H_3\) sharp examples.

\begin{conjecture}
Let \(G\) be a connected subcubic graph.
\begin{enumerate}
\item If \(G\) is bridgeless, then
\(\gamma(G)\leq2\rho(G)+1\).
\item If \(G\) is cubic, then
\(8\gamma(G)\leq17\rho(G)+7\).
\item In general, \(6\gamma(G)\leq13\rho(G)+5\).
\end{enumerate}
\end{conjecture}

\section*{Acknowledgments}

This project originated when the second author, an undergraduate,
used OpenAI Codex, primarily with the GPT-5.6 Sol model, to identify
an open conjecture and explore possible solutions. This preliminary
AI-assisted exploration led to Conjecture~\ref{conj:HLR} and suggested
a candidate counterexample based on path-like assemblies of copies of
the graph now denoted by \(B\). After independently verifying this
candidate, the authors used further computational experiments to
develop the initial construction into the binary branching families
and stronger asymptotic ratio bounds proved here. OpenAI Codex and
other generative-AI tools also substantially assisted in candidate
exploration, computations, the development of verification code, and
manuscript preparation. The authors independently checked all
mathematical claims and take full responsibility for the contents of
the paper.


\end{document}